\documentclass[12pt]{article}      

\usepackage{graphicx}
\usepackage{url}
\usepackage{hyperref}
\usepackage[mathscr]{euscript}		

\usepackage{amsmath}
\usepackage{amsthm}

\usepackage{amssymb}

\usepackage{caption} % caption for sketches

\usepackage{anysize}

\usepackage{enumerate} 
\usepackage{enumitem}
\usepackage[top=.95in, bottom = 1 in, left=1in, right = .6in]{geometry}

\newtheorem{theorem}{Theorem}[section]
\newtheorem{lemma}[theorem]{Lemma}
\newtheorem{corollary}[theorem]{Corollary}
\newtheorem{proposition}[theorem]{Proposition}
\newtheorem{remark}[]{Remark}

\newtheorem{definition}[theorem]{Definition}

\numberwithin{equation}{section}

\newcommand{\ignore}[1]{{}}

\newcommand{\norm}[1]{\lVert #1 \rVert}

\newcommand{\TV}[1]{{\left\vert\kern-0.25ex\left\vert\kern-0.25ex\left\vert #1 
    \right\vert\kern-0.25ex\right\vert\kern-0.25ex\right\vert}}

\newcommand{\m}[1]{\mathbb{#1}}

\def\st{, \,}

\renewcommand\Re{\operatorname{Re}}
\renewcommand\Im{\operatorname{Im}}

\def\a{\alpha}
\def\b{\beta}
\def\g{\gamma}

\def\d{\delta}

\def\t{\theta}

\def\k{\kappa}

\def\vare{\varepsilon}

\def\hcap{{\rm hcap}}

\date{\today}

\begin{document}
\title{Remarks on Loewner Chains Driven by Finite Variation Functions}

\author{Atul Shekhar%
  \thanks{KTH, Lindstedtsv\"agen 25,
         Stockholm, Sweden.
    Email: \url{atuls@kth.se}}
\and Huy Tran
\thanks{TUB, Strasse des 17. Juni 136, 10623 Berlin, Germany.
    Email: \url{tranvohuy@gmail.com}} 
 \and Yilin Wang
\thanks{ETH, R\"amistrasse 101, 8092 Zurich, Switzerland.
    Email: \url{yilin.wang@math.ethz.ch}}
}

\maketitle
\begin{abstract}

To explore the relation between properties of Loewner chains and properties of their driving functions, we study Loewner chains driven by functions $U$ of finite total variation. Under a slow point condition, we show the existence of a simple trace $\gamma$ and establish the continuity of the map from $U$ to $\gamma$ with respect to the uniform topology on $\gamma$ and to the total variation topology on $U$. In the spirit of the work of Wong \cite{carto-wong} and Lind-Tran \cite{huy-lind}, we also 
obtain conditions on the driving function that ensures the trace to be continuously differentiable.
\end{abstract}



%\tableofcontents

\section{Introduction and Results}

The Loewner's differential equation (abbreviated LDE) was introduced by K. Loewner in the context of the Bieberbach conjecture \cite{Loewner1923} where he studied univalent functions on the unit disc by approximating the image domain by slit domains. LDE turned out to be an instrumental tool in its solution which was eventually settled by L. de Branges \cite{DeBranges1985}, see also \cite{duren1977,hayman1994multivalent} for background. 
For those purposes, it was sufficient to consider cases where the slit is an analytical curve which is equivalent to the corresponding driving function being analytical, see \cite{EE01}. The development of the theory of Schramm-Loewner evolution in recent past years has prompted to consider driving functions which are not smooth and to understand the relation between properties of the slit (also called the trace) in terms of properties of its driving function. 

The LDE was initially written in the radial setting where the target point is in the interior of the domain. There exists an equivalent chordal version where the target point is on the boundary of the domain. In the present article, we choose to work with the chordal case, but everything could easily be rephrased in the radial setting. Let us briefly recall some basics of chordal Loewner's theory in the upper half plane $\mathbb {H}:=\{ z | z \in \mathbb{C}, \Im(z) > 0\} $. 

Let $\gamma$ be a continuous injective curve from the compact time interval $[0,T]$ into $\mathbb{H}\cup \{0\}$ with $\g(0) =0$. LDE provides a way to encode the curve $\g$ via a real valued function $U$ which will be called the \textit{driving function} or simply, the \textit{driver} of $\g$. 
Let us first explain how to define the driver $U$ when one knows $\g$. Note that for each $t\geq 0$, $H_t := \mathbb{H}\setminus \g[0,t]$ is a simply connected domain, and there exists a unique conformal map $g_t$ from the slit domain $H_t$ onto $\mathbb {H}$ satisfying the so called \textit{hydrodynamic normalization} given by $\lim_{z \to \infty} (g_t(z) - z  ) =  0$.  
The map $g_t$ will be referred to as the mapping-out function of the set $K_t := \gamma[0,t]$. Further expanding $g_t$ at infinity, one gets the existence of a non-negative constant $b_t$ depending on $K_t $ such that 
\[g_t(z) = z + b_t/z + O(1/|z|^2).\]

The constant $b_t$ is called the half-plane capacity of $K_t$ and is denoted by $b_t = \hcap(K_t)$. It is shown that $t \mapsto \hcap (K_t)$ is continuously increasing. Thus, it is possible to choose a parametrization of $\gamma$ so that $\hcap(K_t) = 2t$ for all $t \in [0,T]$. The mapping-out function $g_t$ also admits a continuous extension to the boundary point $\gamma_t $ of the domain $H_t$. The driver $U$ is then defined by $U_t := g_t(\gamma_t)$ which can be  shown to be a continuous real valued function.  The significance of the driver $U$ comes from the fact that it describes the evolution of the conformal maps $g_t(z)$ in variable $t$ via LDE given by  
\begin{equation}\label{LDE} 
\dot{g}_t(z) = \frac{2}{g_t(z)- U_t}, \hspace{2mm} g_0(z)= z. 
\end{equation} 

In fact, one can also recover the curve $\gamma$ from $U$ as follows. For each $z\in \overline{\mathbb{H}}\setminus \{0\}$, let $[0,T_z)$ with $T_z \in (0, \infty]$ denote the maximal interval of existence of the unique solution to equation \eqref{LDE}. Also define $T_0 = 0$. Then 
\[ \gamma[0,t] = \{z \in \overline{\mathbb{H}} \st T(z) \leq t\}.\]

The procedure described above can also be naturally reversed. Given any continuous real valued curve $U$ with $U_0=0$, written $U \in C_0[0,T]$ hereafter, define $g_t(z)$ for $z \in \overline{\mathbb{H}}\setminus \{0\}$ to be the solution of \eqref{LDE}. Let $T_z$ for $z\in \overline{\mathbb{H}}$ be similarly defined as above. Then 
\[ K_t := \{z \in \overline{\mathbb{H}} \st T(z) \leq t\}\]
defines an increasing family of compact sets in $\overline{\mathbb{H}}$. The family $K= \{K_t\}_{t \in [0,T]}$ is called the \textit{Loewner chain driven by} $U$. As in the previous case, $H_t :=\mathbb{H}\setminus K_t$ is simply connected and $g_t$ is the unique conformal map from $H_t$ onto $\mathbb{H}$ satisfying the hydrodynamic normalization. The Loewner chain $K$ also satisfies $\hcap(K_t)=2t$ and the so-called \textit{conformal local growth property} meaning that the radius of $g_{t}(K_{t+s} \setminus K_t)$
 tends to $0$ as $s \to 0+$ uniformly with respect to $t$. However, it is important to stress that, in general, $K_t$ may not be locally connected, and in full generality, it cannot always be written as the image $\gamma[0,t]$ for a curve $\gamma$. Even if this is the case, the curve $\gamma$ may be non-simple and $K_t$ has to be described by filling the loops in the image $\gamma[0,t]$. We say that the Loewner chain $K$ driven by $U$ admits a trace, or synonymously, $U$ generates a \textit{trace} if there exists a  curve $\g:[0,T]\to \overline{\mathbb{H}}$ such $\gamma_0=0$ and for all $t \in [0,T]$, $H_t$ is the unbounded component of $\mathbb{H}\setminus \gamma[0,t]$. We then call $\gamma$ the trace of the Loewner chain $K$. There are  examples where $K$ does not admit a trace. These cases are of interest too but not the topic of this article. The following questions arise naturally in this context:  

\begin{enumerate} 
\item For what classes of drivers $\mathcal{U}\subset C_0[0,T]$ does the Loewner chain $K$ driven by $U\in \mathcal{U}$ admit a simple trace?
\vspace{2mm}
\item What continuity properties does the map $\Psi$ which maps $U$ to $\gamma$ satisfy on $\mathcal{U}$?
\vspace{2mm}
\item How does the regularity of the trace $\gamma$ relate to properties of the driver $U\in \mathcal{U}$?
\vspace{2mm}
\end{enumerate}

Let us list some answers to the above questions. Marshall-Rohde \cite{RM} and Lind \cite{Lind} have shown that the simple trace exists if $U$ is $1/2$-H\"older with $\sigma_0 = \|U\|_{1/2} <4$. In fact in this case, $U$ generates a quasi-slit; also see \cite{rtz} for a different proof of this theorem. The condition $\sigma_0< 4$ is sharp, and it was shown in \cite{LMR10} and \cite{RM} that there exists a driver $U_t \sim 4\sqrt{1-t}$ as $t \to 1-$,  which does not generate a trace. For $\sigma < 4$, if $\mathcal{U}_{\sigma} := \{U\in C_0[0,T] \st \|U\|_{1/2} \leq \sigma\} $, then the continuity of $\Psi : \mathcal{U}_{\sigma} \to C([0,T], \overline{\mathbb{H}})$ was established in \cite{LMR10} w.r.t. the
 uniform topologies on $\mathcal{U}_{\sigma}$ and $C([0,T], \overline{\mathbb{H}})$. A similar result was obtained in \cite{NikeSun-Scott} without assuming $\|U\|_{1/2}<4$ but instead assuming some technical conditions on limiting trace curve $\gamma$. Some stronger continuity results were obtained in \cite{semimartingale} under the assumption of $U$ being of finite energy, i.e. $\dot{U}$ is square integrable.
%Please see  \cite{NikeSun-Scott}, \cite{Huy}, \cite{Joh15} and \cite{contkappa} for some positive results for continuity properties of $\Psi$ in the random setting.  

%Another line of thought is to notice that the uniform topology on $C_0[0,T]$ for the driving functions corresponds naturally to the Caratheodory topology for the Loewner chain (see section $4.7$ in \cite{lawlerbook} and section $3$ in \cite{NikeSun-Scott}).
%More precisely, for given Loewner chains $K$ and $\tilde{K}$, the Caratheodory distance between them is defined as  
%\[ d_C(K, \tilde{K}) := \sup_{t \in [0,T]} d(g_{K_t}^{-1}, g_{\tilde{K}_t}^{-1}),\]
%where the metric $d$ generates the topology corresponding to  uniform convergence on compact subsets of the upper half-plane.
%The Caratheodory topology for the Loewner chain is the uniform convergence of of $ K_{[0,T]}$  under $d_C$. Instead of measuring directly ``intuitive distance'' between two 
%compacts themselves, Caratheodory distance measures how far are the mapping out functions of them. 
%In the subset of drivers which generate simple curve, it is also natural to equip the family of capacity-parameterized curve in $\m{H}$ with the uniform norm, and study the relationship between the Carath\'eodory metric and this uniform norm. 

 Sufficient conditions on $U$ to ensure certain regularity of $\g$ was obtained by Wong \cite{carto-wong} and Lind-Tran \cite{huy-lind} where it was shown that $ t \mapsto \gamma_{t^2}$ is a $C^{\alpha + 1/2}$ curve when $U$ is $C^{\alpha}$ for $\alpha > 1/2$ (there is a little caveat when $\alpha - 1/2 \in \m N$ though). When $U\in \mathcal{U}_{\sigma}$ for $\sigma <4 $, it was shown in \cite{rtz} that $\gamma$ is $\eta$-H\"older where $\eta$ depends only on the $\sigma$. When $U$ is of finite energy, it was shown in \cite{semimartingale} that $t\mapsto \gamma_{t^2}$ is a Lipschitz curve, and thus $\gamma$ is of finite total variation.

In this article, we prove some further results in the context of the above raised questions. We will more precisely focus on drivers with  finite total variation. Recall that the \textit{total variation} $\TV{U}_{I}$ of a function $U$ on a closed interval $I$  is the supremum of the sum of the absolute values of the increments of $U$ over all partitions  of $I$. We will write $\TV{U}_t := \TV{U}_{[0,t]}$. 
Let us define the following two conditions: 
%\vspace{5mm}
\begin{align*}
  \mbox{(C1):}& \hspace{2mm} \mbox{For all $t >0$,} 
\hspace{2mm} \limsup_{s\to 0+} \frac{ \TV{U}_{[t-s, t ]}  }{\sqrt{s}} < 2.\\
 \mbox{(C2):} &\hspace{2mm} \int_{0+}^{\vare \wedge t} \frac{1}{\sqrt{r}} d \TV U_{t-r} \to 0 \hspace{2mm} \mbox{uniformly for $t \in (0,T]$ as $\vare \to 0$}. 
 \end{align*}

Define the subset $BV_{LR}[0,T] \subset C_0[0,T]$ by (LR stands for ``Locally Regular"),
\[ BV_{LR}[0,T] := \bigl\{ U \in C_0[0,T] ~~ \mbox{ s.t. } ~~\TV{U}_T  <\infty ~~\mbox{  and (C1) holds}\bigr\}. \]

We also equip $BV_{LR}[0,T]$ with the metric $d$ defined by $d(U, V) := \TV{U-V}_T$ for $U, V \in BV_{LR}[0,T]$. Note however that $BV_{LR}[0,T]$ is not a complete metric space. The space $C([0,T], \overline{\mathbb{H}})$ of continuous $\overline{\mathbb{H}}$-valued curves is equipped with the supremum norm hereafter.\\

Our first main result is the following theorem.

\begin{theorem}\label{bvresult}
For each $U \in BV_{LR}[0,T]$, the Loewner chain driven by $U$ admits a simple trace $\gamma$ such that $\g_t \in \mathbb{H}$ for all $t >0$.
\end{theorem}

Our proof is based on a result due to Rohde and Schramm \cite {BasicSLE} which states that the trace exists if and only if 
\begin{equation}\label{cond_RS}
\gamma_t := \lim\limits_{y \rightarrow 0+} f_t (iy + U_t)
\end{equation} 
exists and is continuous in $t$, where $f_t = g_t^{-1}$. If so, the curve $\gamma$ is the trace generated by $U$. We will verify the conditions of \eqref{cond_RS} by providing a candidate for the curve $\gamma$ by uniquely solving reverse time LDE starting from a singularity, see section \ref{Complex-Bessel-Equations} for details. As a result of this, we also obtain the following continuity result.

\begin{theorem} \label{continuity} 
The map $\Psi: (BV_{LR} [0,T],\TV{\cdot}) \to (C([0,T], \overline{\mathbb{H}}),  \|\cdot\|_{\infty})$ defined by $\Psi(U)= \gamma$ is continuous.
\end{theorem}

In fact, we prove Theorem \ref{continuity} under a slightly weaker condition. See section \ref{proof of continuity} for details. Finally, we prove the following result on regularity of the trace $\gamma$.

\begin{theorem}\label{smoothness}
Let $U \in BV_{LR}[0,T]$ such that \mbox{(C2)} holds. Then the curve $t \mapsto \gamma_{t^2}$ is continuously differentiable.
\end{theorem}

Let us make some comments about the above results. Condition (C1)  is reminiscent of a local $1/2$-H\"older condition. But, as we will see in the following, there are examples of functions in $BV_{LR} [0,T]$ that are not $1/2$-H\"older.
In fact, if the driver $U$ is non-decreasing, 
then (C1) is equivalent to saying that for all $t> 0$, there exist $s_0(t) > 0$ and $c(t) <2$ such that for $ s\in (0,s_0(t)]$, $ |U_t - U_{t-s}| \leq c(t) \sqrt{s}$.
Note that this conditions only imposes a $1/2$-H\"older type behaviour from the left at each $t$ with no uniformity assumption on $s_0(t)$ and $c(t)$ with respect to $t$ 
(even though $c(t)$ is assumed to be smaller than $2$, it can get arbitrarily close to $2$ as $t$ varies). 
In \cite[Theorem 1.2]{rtz}, a condition which is very similar to (C2) above assumes that for some constant $C_0$ small enough, 
\begin{equation*}
%\label{condition-huy}
\sup_{t \in (0,T)} \int_0^{t} \sup_{r\in [s,t]} \frac{|U_r - U_s|}{(t-s)^{3/2}}ds \hspace{2mm} \leq C_0,
\end{equation*}
which ensures that the trace is the graph of a Lipschitz function. This can also be compared with Theorem \ref{smoothness} above.

%In the section \ref{(C1)}, in comparison to Theorem $1.2$ in \cite{rtz}, we will make a similar uniformity assumption like \eqref{condition-huy} to produce a continuously
%differentiable trace, which in particular imply the trace to be a graph of Lipschitz function or equivalently a bounded variation curve. We also emphasize that the ``local"
%conditions (C1) and (C2) doesn't imply a global conditions like finite $1/2$-H\"older norm. In the next example, we provide some toy cases of the drivers $U$ to support this claim.  

Let us go through a list of some examples:
\begin{itemize}

\item If $U$ is in the Sobolev space $\mathcal{W}^{1,p}$ with $1\leq p<2$, then $U$ has finite variation and finite $(1-1/p)$-H\"older norm. However, this is not enough to say that $U$ satisfies the condition (C1).

\item On the contrary, if $U$ is in $ \mathcal{W}^{1,2}$, then $U$ is in $BV_{LR}$ and has small $1/2$-H\"older norm on intervals of small length. The latter together with the results in \cite{RM} imply the existence of the trace. This case was also treated in \cite{semimartingale} producing some additional properties of trace such that $t \mapsto \gamma(t^2)$ is a Lipschitz curve, and thus the trace has finite length; see Theorem $2$ in \cite{semimartingale}.

\item For any constant $c$, $U_t = c\sqrt{t}$ (note than when $|c| \geq 4$, then $\|U\|_{1/2} \geq 4$) can be easily seen to be an  element of $BV_{LR}$. 
Note that a scaling argument immediately shows that the trace in this case is a straight line in $\mathbb{H}$ starting at $0$ making an angle $\theta(c)$ with the real axis; also see \cite{knk} and \cite{LMR10} for exact computations. A function like $U_t = 4 \sqrt{t} - 2 \sqrt{t}\log(t)$ is differentiable on $(0,T]$ and clearly $U \in BV_{LR}$. However $\|U\|_{1/2} = + \infty$.

\end{itemize}

Even though the last two examples do not fall in the $\|U\|_{1/2} < 4$ regime, the only problem lies at time $t \to 0+$. One can also instead verify the existence of trace by looking at Loewner chain $\tilde{K}_{t}^{\epsilon}:= g_{\epsilon}(K_{t+ \epsilon} \setminus K_{\epsilon})$. 
Since $\tilde{K}^{\epsilon}$ is driven by $\tilde{U}^{\epsilon}_t = U_{t+ \epsilon} - U_{\epsilon}$ which is continuously differentiable for any $\epsilon >0$, it can be easily seen that $\tilde{K}^{\epsilon}$ admits a trace in $\m H$. 
Finally the conformal local growth property implies that $K$ also admits a trace. A key point to note here is that the
pathological behaviour from the right side of a point can be handled as above. Below we provide some
other examples where we have pathological behaviour from the left side of a point. 
As evident from conditions (C1) and (C2), our approach stresses to control the pathological behavior of $U$ from the left of a time $t>0$. Note that such
a distinction between left and right sides is due to the directional nature of the theory of Loewner
chains.

\begin{itemize}

%
%\item Let $0< \beta <1/3 $, $\b+ 1/2 < \a < 1 - \b/2$. Define
%\begin{equation*}
%U_t =  (1- t)^\a \sin\bigl( (1-t)^{-\b}\bigr) - \sin(1) \quad  \mbox{for} \hspace{2mm} t \in [0, 1)  ,
%\end{equation*}
%and $U_1 = - \sin(1) $.
%Note that for all $t \in (0,1]$, \[ \limsup_{s \uparrow t } \frac{|U_t - U_s|}{|t-s|^{\alpha}} < \infty.\]
%We check the condition (C1): let $t_n = 1 - [\pi (n-\frac{1}{2})]^{-1/\beta}$ be the local extrema of $U$. One can sum the increments of $U$ along the intervals between the points $t_n$ to obtain condition (C1) at $t = 1$ ($U$ is smooth on $[0,1)$):
%\[\frac{\TV{U}_{[t_n, 1]} }{\sqrt{1-t_n}} \sim c  n^{\frac{\b +1/2 -\a}{\b}} \to 0.\]
%However, $\norm{U}_{1/2} = +\infty$ since
%\[\frac{|U(t_{n+1}) - U(t_{n})|}{\sqrt{|t_{n+1} - t_n|}} \sim  c n^{\frac{1- \a-\b/2 }{\b}} \to \infty.\]
%This example shows that local regularities doesn't imply any global regularity. Here $U$ is in
%fact locally better than $1/2$-H\"older but not globally $1/2$-H\"older.
% 

\item \textbf{A monotone $BV_{LR}$ function with infinite $1/2$-H\"older norm :} Let $ c \in (0, 1)$ and $\alpha> 1/2$. Define a sequence by $s_0 = 0$ and $s_n = 1 - c^n$. Note that $s_n \uparrow 1$. 
Choose a strictly increasing sequence $x_n$ with $x_0 = 0$, $x_n \uparrow x$ for some $x$ such that $x - x_n \leq (1-s_n)^{\alpha}  = c^{n\alpha}$. Further choose $t_n \in  (s_n, s_{n+1})$ close enough to $s_n$ so that $(t_n - s_n)^{1/2- \epsilon} < x_{n+1} - x_n$. Now on the sequence $s_0 < t_0 < s_1< t_1 < s_2 < t_2 < ...$, define $U_{s_n} = x_n$, $U_{t_n} = x_n + (t_n - s_n)^{1/2- \epsilon}$ and $U_1 = x$. Interpolating between these points using straight lines gives a continuous monotonic increasing curve. Clearly for $t \in (0,1)$, \[ \limsup_{s \uparrow t } \frac{|U_t - U_s|}{|t-s|^{\alpha}} < \infty.\]
At $t =1$, for $s \in [s_n, s_{n+1}],$ 
\[ U_1 - U_s \leq U_1 - U_{s_n} = x - x_n \leq (1-s_n)^{\alpha} = \frac{(1- s_{n+1})^{\alpha}}{c^\alpha} \leq \frac{(1- s)^{\alpha}}{c^\alpha} \]
which implies \[ \limsup_{s \uparrow 1 } \frac{|U_1 - U_s|}{|1-s|^{\alpha}} < \infty.\]
This also clearly implies conditions (C1) and $ U \in BV_{LR}$. Finally note that 
\[ \frac{|U_{t_n} - U_{s_n}|}{\sqrt{t_n -s_n}} = \frac{1}{(t_n - s_n)^{\epsilon}}\]
and $\|U\|_{1/2} = + \infty$.
\end{itemize}

At last, we mention the following side remark which was the initial motivation to carry out this work. 
In the random setting, Rohde and Schramm \cite{BasicSLE} showed that if $U_t = \sqrt{\kappa}B_t$ where $\kappa >0, \kappa \neq 8$ and $B$ is standard Brownian motion, then almost surely the Loewner chain driven by $U$ admits a trace $\gamma$ (referred as SLE$_\k$). 
Further, $\gamma$ is a simple curve when $\k \leq 4$. In an attempt to understand the sample path properties of $B$ which implies the existence of a simple trace for SLE$_{\kappa}$, $\kappa \leq 4$, a condition like (C1) seems natural since they do not require uniformity with respect to $t$ as explained above.
Even though the Brownian drivers are far from being treated by methods of the present article, Brownian sample paths do satisfy a local regularity condition similar to (C1) at its slow points.
Recall that $t > 0$ is called a $\alpha$-slow point from left for Brownian motion $B$ if 
\[ \limsup_{s \to 0 + } |B_t - B_{t-s}| /\sqrt{s} \leq \alpha. \]
It is well known that such times $t > 0$ exist if $\alpha >1$ and form a dense subset; see e.g. \cite{bmbook}.  Thus, if $\kappa<4$, $\alpha\in (1, 2/\sqrt{\kappa})$ and $t$ is a $\alpha$-slow point of $B$, then
\begin{equation} \label{U-slow}
\limsup_{s \to 0 + } |U_t - U_{t-s}| /\sqrt{s} \leq \sqrt{\kappa}\alpha < 2,
\end{equation}
which is similar to the condition (C1) presented above. 
This is coherent with the fact that SLE$_{\kappa}$ is a simple curve only for $\kappa \leq 4$, suggesting that the constant $2$ appearing in condition (C1) is optimal. 
%It is also worthwhile to note that using results of \cite{rtz}, condition \eqref{U-slow} immediately implies that $\gamma_t = \lim\limits_{y \to 0+}f_t(iy + U_t)$ exists.   
%However, continuity of $\gamma_t$ in $t$ is a more subtle issue.

%YILIN: I FIND THEM DEFOCUSING IN THE REASONING HERE, WE DON'T NEED TO SAY EVERYTHING WE KNOW IN THE INTRO.
Another fact is that the set of slow points is preserved under shifts in the Wiener space by Cameron-Martin $\mathcal{W}^{1,2}$ functions. The class of functions $BV_{LR}$ is also stable under such shifts. We believe that slow points play a crucial (but not complete) role in the existence of trace and it is interesting to look for more deeper properties of Brownian sample paths required to understand the existence of trace for SLE$_{\kappa}$.

%$BV_{LR} + \mathcal{W}^{1,2} = BV_{LR}$, i.e. if $U \in BV_{LR}$ and $\tilde{U}\in \mathcal{W}^{1,2}$, then $U + \tilde{U} \in BV_{LR}$. The space $\mathcal{W}^{1,2}$ is also called Cameron-Martin space in connection to Wiener measure and is of great importance in its study. The fact that space $BV_{LR}$ remain invariant under Cameron-Martin pertubations is very reminiscent of the Cameron-Martin Theorem about Wiener measure. 

The organization of the paper is as follows. In section \ref{Complex-Bessel-Equations} we give the proof of Theorem \ref{bvresult}.  The existence of limit $\gamma_t$ is established in section \ref{Complex-Bessel-Equations-Part-I} and the continuity of $\gamma_t$ is proved in section \ref{Complex-Bessel-Equations-Part-II}. Section \ref{proof of continuity} and section \ref{(C1)} contain the proofs of Theorem \ref{continuity} and Theorem \ref{smoothness} respectively.

{\bf Acknowledgements:} We would like to thank Steffen Rohde and Peter Friz for various discussions and Wendelin Werner for his  valuable comments on an earlier draft of this article. A part of the research was carried out when the first author was a PhD student at Technical University of Berlin and subsequently a visiting fellow at Indian Statistical Institute, Bangalore.  The second author acknowledges the financial support from the European Research Council (ERC) through a Consolidator Grant \#683164.  The third author is supported by SNF grant \#155922. We sincerely thank anonymous reviewers for their helpful comments.

\section{Proof of Theorem \ref{bvresult}}\label{Complex-Bessel-Equations}

 In this section we consider a $U\in BV_{LR}[0,T]$ and employ elementary tools of analysis and measure theory to verify the condition \eqref{cond_RS} which implies existence of the trace. To this end, we subdivide the proof into two parts as follows. The subsection \ref{Complex-Bessel-Equations-Part-I} will be aimed at establishing the existence of the limit 
\begin{equation} \label{RS1} \gamma_t  := \lim\limits_{y \to 0+ } f_t(iy + U_t)
\end{equation} 
 and the subsection \ref{Complex-Bessel-Equations-Part-II} will be aimed at establishing the continuity of the curve $ t \mapsto \gamma_t$.

\subsection{Reverse time Loewner differential equation}  \label{Complex-Bessel-Equations-Part-I}

The basic idea in this section is to utilize reverse time LDE in order to prove the existence of the limit \eqref{RS1}. More precisely, reverse LDE characterizes the dynamics of $f_t(z) $ for $ z \in \mathbb{H}$ as follows. Define $\beta_s^t = U_t - U_{t-s}$ for $ s \in [0,t]$. We fix $t \in (0,T]$ for the rest of this section and with a slight abuse of notation, write $\beta_s$ to mean $\beta_s^t$.

\begin{lemma}\label{inverseflow}
For each fixed $ t \in (0,T]$ and $z \in \mathbb{H}$, 
\[ f_t(z + U_t) = h_t(z)\]where $ h_s(z)_{s \in [0,t]}$ is given by the solution of the reverse time LDE 
\begin{equation}\label{crucialode}
dh_s(z) = d\beta_s + \frac{-2}{h_s(z)}ds, \hspace{2mm} h_0(z) = z.
\end{equation}
\end{lemma}

\begin{proof}Note that $h_s(z) = g_{t-s}(f_t(z+ U_t)) - U_{t-s}$ for $s \in [0,t]$ is a flow from $z $ to $ f_t(z + U_t)$ and using LDE \eqref{LDE}, $h_{s}(z)$ satisfies equation \eqref{crucialode}.
\end{proof} 

Since $z \in \mathbb{H}$, the solution $h_s(z)$ of equation \eqref{crucialode} stays in $\mathbb{H}$. For analysing the behaviour of $h_s(z)$ as $z \to 0$, it becomes beneficial to look at the curves defined by $\phi_s(w) := h_s(\sqrt{w})^2$ for $ w\in \mathbb{C}\setminus [0, \infty)$. Recall the map $z \mapsto z^2$ is a conformal isomorphism $\mathbb{H}\to \mathbb{C}\setminus [0, \infty) $ with the inverse map $ \mathbb{C}\setminus [0, \infty) \to \mathbb{H}$ given by $w \mapsto \sqrt{w}$, where $\sqrt{w}$ is taken to be the square root of $w$ with positive imaginary part. Since $\beta$ is of finite total variation, it easily follows that $\phi_s(w)$ satisfies 
\begin{equation}\label{eqn-of-phi}
 d\phi_s(w) = 2 \sqrt{\phi_s(w)} d\beta_s - 4ds, \hspace{2mm} \phi_0(w) = w
\end{equation}
for each $w \in \mathbb{C}\setminus [0,\infty)$. The key idea here is to give meaning to the curve $\phi_s(0)$ as a solution of the equation \eqref{eqn-of-phi} with starting point $w = 0$. We first need the following definition.

\begin{definition} For a  curve $ X: [0,T] \rightarrow \mathbb{C}$, a branch square root of $X$ is a measurable function $A:[0,T] \rightarrow \overline{\mathbb{H}} $ such that for all $t$, $ A_t ^2 = X_t $.  

\end{definition}

It is easy to check that for any curve $X$, a branch square root exists. Whenever the curve $X$ hits the positive real axis, a branch square root makes a choice of positive or negative square root in a measurable way. Clearly $X$ can have more than one branch square roots in general. With an abuse of notation, we will denote all branch square roots (or a particular one) by symbol $ A_t = \sqrt{X_t}^b$. Note that for any such branch square root, one has $|\sqrt{X_t}^b| = \sqrt{|X_t|}$, and thus $|\sqrt{X_t}^b|$ is continuous. The following lemma will be useful to choose branch square roots which are continuous. First we recall without proof the following basic result which we will frequently use.

\begin{lemma} 
\label{analysis}Let $x_n$ be a sequence in a metric space $M$ and $x\in M$ an element such that for any subsequence $x_{n_k}$, there is a further  subsequence %$x_{n_{k_l}}$ 
which converges to $x$. Then the sequence $x_n$ converges to $x$.

\end{lemma}

\begin{lemma} \label{limitbranch} Let $X^n , X : [0,T] \rightarrow \mathbb{C}$ be  curves with $X_0 = 0 $, $ X_0^n \in \mathbb{C}\setminus(0, \infty)$ and $ X_t^n \in \mathbb{C}\setminus [0, \infty)$ for all $n$ and $t > 0$. If $X^n$ converges uniformly to $X$, then there exist a branch square root $\sqrt{X}^b$ of $X$ and a subsequence $X^{n_k}$ such that $\sqrt{X^{n_k}}$ converges uniformly to $\sqrt{X}^b$. In particular, $\sqrt{X}^b$ is continuous. Further, if $X_t \in \mathbb{C}\setminus [0, \infty)$ for all $t > 0$, then $\sqrt{X^n}$ converges uniformly to $\sqrt{X}$.
\end{lemma}

\begin{proof} Note that family of curves $\{\sqrt{X^n}\}$ is uniformly bounded. We will prove that this family is equicontinuous. Then the Arzela-Ascoli's theorem implies that there exists a subsequence $\sqrt{X^{n_k}}$ converging uniformly to a  continuous function $A$ which is a branch square root of $X$. 

For proving the equicontinuity of the family $\{\sqrt{X^n}\}$, let $\epsilon > 0$. We need to exhibit a $\delta$ such that if $|t-s| \leq \delta$, then $|\sqrt{X_t^n} - \sqrt{X_s^n}| = O( \epsilon )$  for all $n$.
Since the family $\{X^n\}$ is equicontinuous, first choose $\delta>0$ such that for all $n$, 
$$|X_t^n-X_s^n|\leq \epsilon^2 \mbox{ whenever } |t-s|\leq \delta.$$
%One of the key ideas is the identity $$|\sqrt{X_t^n} - \sqrt{X_s^n}| = \frac{|X_t^n - X_s^n|}{|\sqrt{X_t^n} +  \sqrt{X_s^n}|}$$
%The problem occurs when $|X^n_t|$ and $|X^n_s| are both large, $X^n_s$ is close to $X^n_t$, and when  when $$\sqrt{X^n_t}$$ and $\sqrt{X^n_s}$ have the opposite real parts and small imaginary parts. 
%When that happens, $X^n_t$ and $X^n_s$ on two different sides of $[0,\infty)$$. This gives a contradiction since it takes a long distance to travel from $X^n_t$ to $X^n_s$ along the curve $X^n$.
Fix $t$ and $s$ such that $|t-s|\leq \delta$. If $|X^n_t|\leq \epsilon^2$ or $|X^n_s| \leq \epsilon^2$, then both $|X^n_t|\leq 2\epsilon^2$ and $|X^n_s|\leq 2\epsilon^2$, which implies $|\sqrt{X_t^n} - \sqrt{X_s^n}| \leq 2\sqrt{2}\epsilon $.

Otherwise, that is if $|X^n_t|\geq \epsilon^2$ and $|X^n_s|\geq \epsilon^2$, then we claim that $|\sqrt{X^n_s} + \sqrt{X^n_t}|\geq \epsilon/2$, which by the identity 
$$|\sqrt{X_t^n} - \sqrt{X_s^n}| = \frac{|X_t^n - X_s^n|}{|\sqrt{X_t^n} +  \sqrt{X_s^n}|}$$
 implies that $|\sqrt{X_t^n} - \sqrt{X_s^n}| =O(\epsilon)$.

Indeed, since 
$$\epsilon^2\leq |X^n_s| = \Re X^n_s + 2(\Im \sqrt{X^n_s})^2, $$
if $\Re X^n_s\leq \epsilon^2/2$, then 
$$ |\sqrt{X^n_s} + \sqrt{X^n_t}| \geq \Im \sqrt{X^n_s} \geq \epsilon/2. $$
It is similar if $\Re X^n_t\leq \epsilon^2/2$. Now, if both $\Re X^n_t> \epsilon^2/2$ and $\Re X^n_s> \epsilon^2/2$, then the signs of $\Im(X_s^n)$ and $\Im(X_t^n)$ are the same (since the curve $X^n$ does not intersect the positive real axis). That implies $\Re\sqrt{X^n_t}$ and $\Re\sqrt{X^n_s}$ have the same signs. Hence,
$$ |\sqrt{X^n_s} + \sqrt{X^n_t}| \geq \Re \sqrt{X^n_s} \geq \epsilon/\sqrt{2}. $$
The above mentioned claim is proved.

%W.l.o.g. we can assume $|X_s^n| \geq \epsilon^2/4$. Since family $\{X^n$\} is equicontinuous, choose $\delta$ such that for all $n$, $|X_t^n - X_s^n| \leq \epsilon^2/16 $ and thus $ |X_t^n| \geq \epsilon^2/8$. We claim that $|\sqrt{X_s^n} + \sqrt{X_t^n}| \geq c \epsilon$ for some constant $c >0$. 

%If either $ \Re( X_s^n) \leq \epsilon^2/16$ or $\Re( X_t^n) \leq \epsilon^2/16$, then $\Im(\sqrt{X_s^n} + \sqrt{X_t^n} ) \geq c\epsilon$. 
%On the other hand, if  $ \Re( X_s^n) \geq \epsilon^2/16$ and $ \Re( X_t^n) \geq \epsilon^2/16$, then the signs of $\Im(X_s^n)$ and $\Im(X_t^n)$ are the same (since the curve $X^n$ does not intersect the positive real axis) and thus $ | \Re(\sqrt{X_s^n} + \sqrt{X_t^n} )| \geq c\epsilon$ . Finally, \[|\sqrt{X_t^n} - \sqrt{X_s^n}| = \frac{|X_t^n - X_s^n|}{|\sqrt{X_t^n} +  \sqrt{X_s^n}|} \leq c \epsilon \]for some $c >0$. 

Finally, if $X_t \in \mathbb{C}\setminus [0, \infty)$ for $t>0$, then there is only one branch square root given uniquely by $\sqrt{X}$. The uniform convergence of the whole sequence is a consequence of Lemma \ref{analysis}.
\end{proof}

We are now ready to give a sense to equation \eqref{eqn-of-phi} with  $w =0$. A curve $\phi_s = \phi_s(0)$ is called a solution to \eqref{eqn-of-phi} with $w=0 $ if 
\begin{equation}\label{w=0} \phi_s = 2 \int_0^s \sqrt{\phi_r}^bd\beta_r - 4s 
\end{equation}
for some continuous branch square root $\sqrt{\phi}^b$, where the integral is interpreted as an Riemann-Stieltjes integral.

\begin{remark} We have crucially utilized the assumption that $\beta$ is of finite total variation while giving meaning to the equation \eqref{w=0} because $\sqrt{\phi}^b$ is assumed to be continuous only and integral in \eqref{w=0} is understood as a Riemann-Stieltjes/Lebesgue integral. We believe that the assumption of finite total variation and properties of Riemann-Stieltjes/Lebesgue integral are crucial for the proofs in this paper.
%We implicitly use properties of Riemann-Stieltjes integral in many places.
In Lemma \ref{analysis-II} in the next section \ref{Complex-Bessel-Equations-Part-II}, we will see another important feature of measure theory which is Portamanteau Theorem or weak compactness of totally bounded sets to be crucially important. In particular, we found it non-trivial to avoid the condition of finite total variation and perhaps use other calculus methods e.g. Young's calculus in order to consider drivers of finite $p$-variation for $p>1$. We plan to study such drivers in our future projects.

\end{remark}

Our next goal is to establish existence and uniqueness of solution to equation \eqref{w=0}. To this end we first prove the following lemma.

\begin{lemma} \label{uniquebranch}
Let a  curve $  v : [0,t] \rightarrow \mathbb{C}$ with a continuous branch square root $ \sqrt{v}^b$ satisfy $|\Re(\sqrt{v_s}^b)| \leq \TV{\beta}_s$ and 
\begin{equation}\label{eqn-of-v} v_s = 2\int_0^s \sqrt{v_r}^b d\beta_r - 4s 
\end{equation}
for all $ s \in [0,t]$. If for some $\delta < 2 $ and $s_0 \in (0, t]$ depending on $t$, $\TV{\b}_s \leq  \d \sqrt{s}$ for $s\in [0,s_0]$, then 
\begin{enumerate}
\item For all $s\in (0,t]$,  $v_s \in \mathbb{C}\setminus [0, \infty)$ and $ \sqrt{v_s}^b = \sqrt{v_s}$. 

\item Moreover, for $c_\d = \sqrt{4-\d^2}> 0$ and $ s \in (0,s_0]$, 
\begin{equation}
     \label{lower bound} c_\d \leq \Im(\sqrt{v_s})/\sqrt{s} \leq 2. 
\end{equation}
In particular, (a) and (b) hold if the condition (C1) is satisfied. 

\item There exists a constant $C$ depending only on $\beta$ such that $\|v\|_{\infty}\leq C$ and $|v_r - v_s| \leq C(\TV{\beta}_{[r,s]} + s-r)$ for all $r\leq s$.
\end{enumerate}

\end{lemma}

\begin{proof} The condition $\TV{\beta}_s \leq \d \sqrt{s}$ for every $s\in [0,s_0]$ implies  that for $s \in (0,s_0]$,
\[\frac{2}{s} \int_0^s | \Re(\sqrt{v_r}^b)|  d \TV{\beta}_r \leq \frac{1}{s}\TV{\beta}_s^2 \leq \d^2 <4.  \]
Thus,
\[ \Re(v_s) \leq  (\d^2 -4) s, \]
which implies  for $  s \in (0, s_0]$ that $v_s \in \mathbb{C}\setminus [0, \infty)$ and that
\[\Im (\sqrt{v_s}) \geq \sqrt{4-\d^2 }\sqrt{ s}. \]
 Since the solution of equation \eqref{eqn-of-phi} remains in $\mathbb{C}\setminus [0, \infty)$ once the starting point $ w \in \mathbb{C}\setminus [0, \infty)$, we conclude that $ v_s \in \mathbb{C}\setminus [0, \infty)$ for all time $ s \in (0,t]$.
 
  Write $X_s + iY_s = \sqrt{v_s}$. By comparing the real and imaginary parts on both sides of  equation (\ref{eqn-of-v}), we derive differential formulae for $X^2_s - Y^2_s$ and $2X_sY_s$
  \begin{eqnarray*}
  d(X^2_s - Y^2_s) & = & 2X_s d\beta - 4ds\\
  d(2X_s Y_s) & = & 2Y_s d\beta.
  \end{eqnarray*}
  
  Then we can step-by-step deduce differential formulae for $(X^2_s+Y^2_s)^2$, $X^2_s + Y^2_s$, $X^2_s$, and $Y_s$. In particular, 
   $$d(Y^2_s) = \frac{4Y_s^2}{X^2_s + Y^2_s} ds$$
   which means 
    \[Y_s^2 = \int_0^s \frac{4 Y_r^2 }{ X_r^2+ Y_r^2} d r \leq 4s.\]
  This implies the other inequality in (\ref{lower bound}) and the boundedness of $v$.
  
 % 
 %That implies
 % d(Y^2_s) = \frac{1}{2} d(X^2_s - Y^2_s) + \frac{1}{2} d(X^2 + Y^2_s) = \frac{4Y_s^2}{X^2_s + Y^2_s}

%The other inequality in (\ref{lower bound}) and the boundedness of $v$ follow easily from \[(\Im(\sqrt{v_s}))^2 = \int_0^s \frac{4 \Im(\sqrt{v_r})^2 }{ \Re(\sqrt{v_r})^2+ \Im(\sqrt{v_r})^2} d r \leq 4s.\]

Finally, a bound on the modulus of continuity of $v$ follows by applying triangle inequality to the equation \eqref{eqn-of-v}.
\end{proof}

 The above lemma tells us in particular, under the condition (C1), solutions to the equation \eqref{w=0} leave $[0,\infty)$ immediately and hence the equation \eqref{w=0} can be equivalently written with the usual complex square root as 
\begin{equation}\label{w=0-new}
\phi_s = 2 \int_0^s \sqrt{\phi_r}d\beta_r - 4s. 
\end{equation} 

We now prove the following result on the existence and the  uniqueness of solution to \eqref{w=0-new}. 

\begin{proposition}\label{existenceofphi} %\label{unique}%
Let $U\in BV_{LR}$. Then there exists a unique continuous function $\phi_s = \phi_s(0)$ with $\phi_s \in \mathbb{C}\setminus [0, \infty)$ for $s > 0$, $ |\Re(\sqrt{\phi_s}) | \leq \TV{\beta}_s $, and satisfying \eqref{w=0-new}.

%\begin{equation}\label{phi(0)} \phi_s = 2 \int_0^s \sqrt{\phi_r} d\beta_r - 4s
%\end{equation}
\end{proposition}

\begin{proof}We first address the uniqueness of solution. Let $\phi^1$ and $\phi^2$ be two solutions satisfying the conditions above.  From Lemma \ref{uniquebranch}, for $i=1,2$ and $s \leq s_0$,
\begin{equation}\label{slope-bound}
 \sqrt{4-\delta^2} \leq \frac{\Im(\sqrt{\phi_s^i})}{\sqrt{s}} \leq 2 \hspace{4mm} \mbox{and} \hspace{4mm} \frac{|\Re(\sqrt{\phi_s^i})|}{\Im(\sqrt{\phi_s^i})} \leq \frac{\delta}{\sqrt{4-\delta^2}}.
\end{equation}

%we observe that $\phi_s^i \in \mathbb{C}\setminus [0, \infty)$ for each $s >0$ and $\sqrt{\phi_s^i}^{b_i} = \sqrt{\phi_s^i}$. In fact from the proof of Lemma \ref{uniquebranch}, we have \[ \limsup_{s\rightarrow 0+ } \frac{\Re(\phi_s^i)}{s} < 0 \]

In particular,
\[ \int_{0+}^s \frac{1}{|\sqrt{\phi_r^i|}}dr < \infty\]and 
\[ \sqrt{\phi_s^i} = \beta_s  + \int_{0+}^{s}\frac{-2}{\sqrt{\phi_r^i}}dr.\]

Write $\sqrt{\phi_s^1} = X_s + i Y_s, \sqrt{\phi_s^2} = \tilde{X}_s + i \tilde{Y}_s$. Then for any $0< u <s \leq s_0 $,
\[\sqrt{\phi^1_s} - \sqrt{\phi^2_s} = \sqrt{\phi^1_u} - \sqrt{\phi^2_u} + \int_u^s \frac{2(\sqrt{\phi^1_r} - \sqrt{\phi^2_r})}{\sqrt{\phi^1_r}\sqrt{\phi^2_r}}dr, \]
which implies 
\begin{equation}\label{derivative}
\sqrt{\phi^1_s} - \sqrt{\phi^2_s} = (\sqrt{\phi^1_u} - \sqrt{\phi^2_u}) \exp\biggl[ \int_u^s \frac{2}{\sqrt{\phi^1_r}\sqrt{\phi^2_r}}dr\biggr].
\end{equation}
Note that for $s> 0$, we have $Y_s, \tilde{Y}_s >0$ and 
$$\frac{d\log(Y_s)}{ds} = \frac{2}{X_s^2 + Y_s^2} \,\text{ and }\, 
\frac{d\log(\tilde{Y}_s)}{ds} = \frac{2}{\tilde{X}_s^2 + \tilde{Y}_s^2}.$$
Use the estimate  \eqref{slope-bound},
\begin{align*}
\Re\biggl[ \int_u^s \frac{2}{\sqrt{\phi^1_r}\sqrt{\phi^2_r}}dr\biggr] &= \int_u^s \frac{2(X_r\tilde{X}_r - Y_r\tilde{Y}_r)}{(X_r^2 + Y_r^2)(\tilde{X}_r^2 + \tilde{Y}_r^2)}dr  \\
&\leq \int_u^s \frac{2X_r\tilde{X}_r}{(X_r^2 + Y_r^2)(\tilde{X}_r^2 + \tilde{Y}_r^2)}dr \\
& \leq \int_u^s  \frac{X_r^2}{(X_r^2 + Y_r^2)^2}dr + \int_u^s  \frac{\tilde{X}_r^2}{(\tilde{X}_r^2 + \tilde{Y}_r^2)^2}dr \\
&\leq \frac{\delta^2}{8}\biggl(\int_u^s \frac{2}{(X_r^2 + Y_r^2)}dr + \int_u^s \frac{2}{(\tilde{X}_r^2 + \tilde{Y}_r^2)}dr\biggr) \\
& = \frac{\delta^2}{8}\log\biggl\{\frac{Y_s\tilde{Y}_s}{Y_u\tilde{Y}_u}\biggr\}.
\end{align*}

Thus, it follows from \eqref{derivative}

\begin{equation}\label{final-bound}
|\sqrt{\phi^1_s} - \sqrt{\phi^2_s}| \leq ({Y_s\tilde{Y}_s})^{\delta^2/8}|\sqrt{\phi^1_u} - \sqrt{\phi^2_u}|({Y_u\tilde{Y}_u})^{-\delta^2/8}.
\end{equation}
Since $\Re(\sqrt{\phi_u^i}) \leq \TV{\beta}_u \leq \delta \sqrt{u}$ and $\Im(\sqrt{\phi^i_u}) \leq 2\sqrt{u}$, $|\sqrt{\phi^1_u} - \sqrt{\phi^2_u}| \leq C\sqrt{u}$ for some constant $C$. Also, again from \eqref{slope-bound}, $Y_u\tilde{Y}_u \geq (4-\delta^2)u$. Thus, \eqref{final-bound} gives 
\[|\sqrt{\phi^1_s} - \sqrt{\phi^2_s}| \leq C ({Y_s\tilde{Y}_s})^{\delta^2/8} \sqrt{u}^{1 - \delta^2/4}\]
for some constant $C$. Since $\delta < 2$, by letting $u \to 0+$, the right hand side of the previous inequality tends to $0$ which implies $\phi_s^1= \phi_s^2$ for $s \leq s_0$. Finally, the fact $\phi_{s_0}^1 = \phi_{s_0}^2 \in \mathbb{C}\setminus [0, \infty)$ and the uniqueness of solution to equation \eqref{eqn-of-phi} for starting point $ w \in\mathbb{C}\setminus [0, \infty) $ imply $\phi_s^1 = \phi_s^2$ for all $ s \in [0,t]$. 

For the existence of a solution, by applying triangle inequality to \eqref{eqn-of-phi}, one can see that the functions $\{\phi_\cdot(-y^2), y \in (0,1]\}$ form a uniformly bounded equicontinuous family. Thus, by Arzela-Ascoli Theorem and Lemma \ref{limitbranch}, there is a subsequence $\phi(-y_n^2)$ converging uniformly to a  continuous function $\phi$ and $ \sqrt{\phi(-y_n^2)}$ converging uniformly to some continuous branch square root $\sqrt{\phi}^b$ as $ y_n \rightarrow 0 + $. Then it follows from a convergence theorem of Riemann-Stieltjes integrals that \[ \phi_s = 2 \int_0 ^s \sqrt{\phi_r}^bd\beta_r - 4s.\]
Also, it follows from equation \eqref{crucialode} that if $ X_s + i Y_s = h_s(iy) = \sqrt{\phi_s(-y^2)}$, then 
\begin{eqnarray*}
d(X_s-\beta_s) &=& \frac{-2X_s}{X_s^2+Y_s^2}ds \\
\mbox{ and }\quad dY_s  &=& \frac{2Y_s}{X_s^2+Y_s^2}ds.
\end{eqnarray*}
It implies that 
$$d (X_s Y_s) = X_s dY_s + Y_s dX_s = Y_s d\beta_s.$$
Therefore
$$X_s=\frac{1}{Y_s} \int^s_0 Y_r d\beta_r. $$
Use the Riemann-Stieltjes inequality and the monotonicity of $Y_s$, 
$$|X_s|\leq \frac{1}{Y_s} \left(\sup_{r\in [0,s]} Y_r\right) \,\TV{\beta}_s = \TV{\beta}_s.$$

%\[ X_s = \beta_s  - \frac{1}{Y_s} \int_0 ^s \beta_r d Y_r.\]
In particular, $ |\Re(\sqrt{\phi_s(-y^2)})| \leq \TV{\beta}_s$ for all $y>0$. Thus, $ |\Re( \sqrt{\phi_s}^b)| \leq \TV{\beta}_s$. Finally, use Lemma \ref{uniquebranch} to obtain $\phi_s \in \mathbb{C}\setminus [0, \infty)$ for all $ s > 0$ and $\sqrt{\phi}^b = \sqrt{\phi} $, which concludes the proof.
\end{proof}

As an immediate corollary, we obtain the existence of the limit \eqref{RS1}.

\begin{corollary}\label{limitexist}
The solution $\phi(-y^2)$ of equation \eqref{eqn-of-phi} with $ y>0$ converges uniformly to the solution $ \phi(0)$ as $ y \rightarrow 0+$. In particular, $f_t(iy + U_t) = h_t(iy ) = \sqrt{\phi_t(-y^2)}$ converges to $\sqrt{\phi_t(0)}$ as $ y \rightarrow 0+$.
\end{corollary}

\begin{proof}
As in the proof of Proposition \ref{existenceofphi}, for any sequence $\phi(-y_n^2)$ with $y_n \to 0+$, there is a subsequence $\phi(-y_{n_k}^2)$ converging uniformly to a solution of equation \eqref{w=0-new}. Since $\phi(0)$ is the unique solution of equation \eqref{w=0-new}, using Lemma \ref{analysis}, we conclude that $\phi(-y^2)$ converges uniformly to $\phi(0)$ as $ y \rightarrow 0+$. Finally, since $ \phi_t(0) \in \mathbb{C}\setminus [0, \infty)$, we arrive at $ \sqrt{\phi_t(-y^2)} \rightarrow  \sqrt{\phi_t(0)}$ as $ y \rightarrow 0+$.
\end{proof}

\subsection{Continuity of the map $t\mapsto\gamma_t$} \label{Complex-Bessel-Equations-Part-II}

In this section, we prove the continuity of $\gamma$ defined by equation \eqref{RS1}. At this point we denote the solution constructed in Proposition \ref{existenceofphi} as $ \phi_s^t = \phi_s^t(0)$ for $ s \in [0,t]$. As seen in Corollary \ref{limitexist}, \[ \gamma_t = \sqrt{ \phi_t^t}\]

The following lemma will be the key for establishing the continuity of $\gamma$. 
\begin{lemma}\label{analysis-II}
Let $X^n$ be a sequence of continuous functions on $[0,T]$ converging uniformly to $X$. Suppose $\sup_n \TV{X^n}+ \TV{X} < \infty$, then for any continuous function $Z$,  
\[ \int_0^TZ_rdX_r^n \rightarrow \int_0^T Z_rdX_r \hspace{2mm} \mbox{as} \hspace{2mm} n \to \infty. \]
\end{lemma}

\begin{proof}
The proof follows easily as an application of Portmanteau Theorem and is left to the reader to verify.  
\end{proof}

\begin{proposition}\label{continuity in t} The map $ t \mapsto \phi_t^t $ is continuous. In particular, $\gamma$ is a  curve.
 
\end{proposition}

\begin{proof}
Note that for $s \in [0,t]$, \[ \phi_s^t = 2\int_0 ^s \sqrt{\phi_r^t}d\beta_r^t - 4s  \] 
From Lemma \ref{uniquebranch}, the curves $\phi^t$ are uniformly bounded in $t$ and \[ |\phi_t^t |\leq C(\TV{\beta^t}_{t}+ t) \]
for some constant $C$, implying continuity at $ t = 0$ \[ \lim\limits_{t \rightarrow 0+ }\phi_t^t = 0 .\]
For continuity on $ (0, T]$, fix a time $ t_0 > 0$. Then for $ t \in ( t_0 /2 , 2t_0) $, define $\alpha_s^t = \phi_{ s t/t_0}^t$ for $ s \in [0, t_0]$. Note that  $| \Re(\sqrt{\alpha_s^t})| \leq \TV{\beta^t}_{ st/t_0} $ and \[ \alpha_s^t = 2 \int_0^s \sqrt{\alpha_r^t} d \beta_{ r t/ t_0}^t - 4 s t/t_0.\]
Lemma \ref{uniquebranch} implies that the family of curves $\{\alpha^t\}$ is uniformly bounded and equicontinuous. Again, by Arzela-Ascoli's theorem and Lemma \ref{limitbranch}, along some subsequence $t_n \rightarrow t_0 $, $ \alpha^{t_n} $ converges uniformly to some continuous function $\tilde{\phi}$ and $ \sqrt{\alpha^{t_n}}$ converges uniformly to some branch square root $ \sqrt{\tilde{\phi}}^b$ with $ |\Re(\sqrt{\tilde{\phi}_s}^b)| \leq \TV{\beta^{t_0}}_s$ on $[0,t_0]$. As an application of Lemma \ref{analysis-II} and together with the Riemann-Stieltjes' inequality, we see that \[ \tilde{\phi}_s = 2 \int_0 ^s \sqrt{\tilde{\phi}_r}^bd \beta_r ^{t_0} - 4 s. \] 
 Using Lemma \ref{uniquebranch} and Proposition \ref{existenceofphi}, we conclude that $ \tilde{\phi}_s = \phi_s^{t_0 }$. Finally, Lemma \ref{analysis} implies that $\alpha^{t}$ converges uniformly to $ \phi^{t_0}$ as $  t \rightarrow t_0$. In particular, $ \phi_t^t = \alpha_{t_0}^t \rightarrow \phi_{t_0}^{t_0}$. Note that $ \phi_t^t \in \mathbb{C}\setminus [0, \infty) $ for all $ t > 0 $. Thus, $\gamma_t = \sqrt{\phi_t^t}$ is also a curve.   
\end{proof}

\begin{proof}[Proof of Theorem \ref{bvresult}]
The existence of the trace $\gamma$ follows from Corollary \ref{limitexist} and Proposition \ref{continuity in t}. Clearly, $\gamma_t \in \mathbb{H}$ for all $t>0$ from the construction above of $\gamma$. For the simpleness of $\gamma$, suppose on the contrary $\gamma_s = \gamma_{s'} $ for $ s < s'$. Note that chain $\tilde{K}_t := g_s( K_{t + s}\setminus K_s) - U_s$ is driven by $ \tilde{U}_t  = U_{t+s }- U_s$. Clearly $ \tilde{U}\in BV_{LR}$ and by above argument $\tilde{K}$ is generated by a curve $\tilde{\gamma}$ with $\tilde{\gamma}_t \in \mathbb{H}$. But since $\gamma_s = \gamma_{s'} $, $\tilde{\gamma}_{s'- s} \in \mathbb{R}$ which is a contradiction. 

\end{proof}

\begin{remark}
We emphasize that in our approach it was very beneficial to consider the squared equation for $\phi_s^t(0) = h_s^t(0)^2$ starting from $0$ instead of considering the equation for $h_s^t(0)$ itself with the additional assumption that it takes values in upper half plane $\mathbb{H}$. Even though the solution $h_s^t(0)$ eventually takes value in $\mathbb{H}$, it was important in the  proof above to consider equation \eqref{w=0} which is unconditionally well defined by using the concept of branch square root. This is because $\mathbb{H},\mathbb{C}\setminus [0,\infty)$ are open sets and a sequence of functions taking values in here can escape the set in the limit. Thus, imposing the additional assumption that $h_s^t(0) \in \mathbb{H}$ is not stable when the parameter $t$ is varied. As evident in the proof above, we got around this issue while checking the continuity of $t\mapsto \gamma_t$  by considering \eqref{w=0} instead. 
\end{remark}

\section{Proof of Theorem \ref{continuity}}\label{proof of continuity}

In this section we will employ the approach developed in the previous section to obtain the continuity of the map $\Psi: BV_{LR}[0,T] \to C([0,T], \overline{\mathbb{H}})$ mapping $U \mapsto \gamma$. In fact we will prove a slightly stronger version of Theorem \ref{continuity} as follows. 

\begin{proposition} Let $U^n, U \in BV_{LR}[0,T]$ with $\|U^n - U\|_{\infty} \to 0 $ as $n \to \infty$. Further assume that $\sup_{n}\TV{U^n}_{T} < \infty$ and family of curves $s \mapsto \TV{U^n}_{s}$ is equicontinuous in $n$. If $\gamma^n$ and $\gamma$ are the trace of Loewner chain driven by $U^n$ and $U$ respectively, then 
\[ \|\gamma^n - \gamma\|_{\infty} \to 0 \hspace{2mm} \mbox{as} \hspace{2mm} n \to \infty.\] 

\end{proposition}

\begin{proof}
We will use the notations from Section \ref{Complex-Bessel-Equations}. Let $\phi_s^{n,t}, \phi_s^{t}$ are the solutions to equation \eqref{w=0-new} driven by $\beta^{n,t}, \beta^t$ respectively as produced in Proposition \ref{existenceofphi}. For each $t_0 > 0$ and $ t \in (\frac{t_0}{2} , 2t_0 )$, define $\alpha_s^{n,t} = \phi_{\frac{ts}{t_0}}^{n,t}$  for $s \in [0,t_0]$ and note that 
\[ \alpha_s^{n,t} = 2\int_0^s \sqrt{\alpha_r^{n,t}} d\beta_{r t/t_0}^{n,t} - 4 st/t_0\] 
Lemma \ref{uniquebranch} implies that $\alpha^{n,t}$ is uniformly bounded in $n,t$ and using equicontinuity of $s \mapsto \TV{U^n}_{s}$ in $n$, we see that the family of curves $\{\alpha^{n,t}\}$ is also equicontinuous. Using Arzela-Ascoli Theorem and Lemma \ref{limitbranch}, along some subsequence $(n_k,t_k) \to (\infty, t_0) $, $\alpha^{n_k, t_k}$ converges uniformly to some curve $\tilde{\phi}$ and $\sqrt{\alpha^{n_k, t_k}}$ converges to some branch square root $\sqrt{\tilde{\phi}}^b$ of $\tilde{\phi}$. Using Lemma \ref{analysis-II}, $\tilde{\phi}$ satisfies 
\[ \tilde{\phi}_s = 2 \int_0^s \sqrt{\tilde{\phi}_r}^bd\beta_r^{t_0} - 4s \]
and Lemma \ref{uniquebranch} and Proposition \ref{existenceofphi} implies $\tilde{\phi} = \phi^{t_0}$. A variant of Lemma \ref{analysis} for double indexed sequences implies that $\alpha^{n,t}$ converges uniformly to $\phi^{t_0}$ as $(n, t) \to (\infty, t_0)$. In particular, $ \phi_{t}^{n,t} = \alpha_{t_0}^{n,t} \to \phi_{t_0}^{t_0}$ for each $t_0 > 0$. 
Also for $t_0 = 0 $, since
\[ |\phi_t^{n,t}| \leq C(\TV{\beta^{n,t}}_{t}+ t),\] 
we have $\phi_{t}^{n,t} \to \phi_{0}^0 = 0$ as $(n,t) \to (\infty, 0)$. In other words, for each $\epsilon > 0$ and $t_0 \in [0,T]$, there exists natural number $N_{t_0, \epsilon}$ and open ball $B_{t_0, \epsilon}$ around $t_0$ such that for $n \geq N_{t_0, \epsilon}$ and $t \in B_{t_0, \epsilon}$, 
\[ |\phi_t^{n,t} - \phi_{t_0}^{t_0}| \leq \epsilon. \]

By possibly choosing a smaller radius for ball $B_{t_0, \epsilon}$, we see that 
\[ |\phi_t^{n,t} - \phi_t^t| \leq |\phi_t^{n,t} - \phi_{t_0}^{t_0}| + |\phi_t^{t} - \phi_{t_0}^{t_0}| \leq 2\epsilon.\] 
The collections of balls $\{B_{t_0, \epsilon}\}_{t_0 \in [0,T]}$ forms an open cover of compact set $[0,T]$. Hence, it has a finite subcover, say $\{B_{t_i, \epsilon}\}_{i = 1, .., m}$. Now, for $n \geq \max_{i = 1, .. , m} N_{t_i, \epsilon}$, 
\[\sup_{t\in [0,T]} |\phi_t^{n,t} - \phi_t^t| \leq 2\epsilon\]
implying the uniform convergence of $\phi^{n}$  to $\phi$. Finally, note that $\gamma_t^n = \sqrt{\phi_t^{n,t}}$, $\gamma_t = \sqrt{\phi_t^{t}}$ and application of Lemma \ref{limitbranch} concludes the proof.
\end{proof}

\section{Proof of Theorem \ref{smoothness}}\label{(C1)}

In this section, we provide a sufficient condition on $U \in BV_{LR}[0,T]$ to generate a $C^1$ trace. Along with the assumption 
%that $\|U\|_{1/2} < 2$ and 
$U \in BV_{LR}[0,T]$, we further assume 

\[  \mbox{For all $t >0$,} \hspace{2mm} \int_{0+}^t \frac{1}{\sqrt{r}}d\TV{\beta^t}_r < \infty \]
and that the same integral from $0$ to $\vare$ converges uniformly to $0$ in the following sense:

\[ \mbox{(C2):} \hspace{2mm}  \exists \text{ increasing function } \delta: (0,T] \mapsto \m{R}_+, \text{ s.t. }  \delta(\vare) \xrightarrow[]{\vare \to 0} 0  \]
\[\text{ and } \int_{0+}^{\vare \wedge t} \frac{1}{\sqrt{r}}d\TV{\beta^t}_r \leq \delta(\vare), \forall t \in (0,T].\]

By  probably restricting to a smaller interval $[0,T]$, without loss of generality, we can assume that $\sup_{\vare > 0 } \d(\vare) =: c  < 2$.

\begin{proposition}\label{thm_(C1)} Let $ U \in BV_{LR}[0,T]$.
Further, suppose the condition (C2) holds. Then, the curve $t \mapsto \phi_t^t(0)$ is continuously differentiable. In particular, the curve $t \mapsto \gamma_{t^2}$ is continuously differentiable.

\end{proposition}

Before going into the proof, we list some remarks regarding the condition (C2).
     \begin{itemize} 
        \item Note that (C2) is stronger than condition (C1) appearing in the definition of space $BV_{LR}[0,T]$. It can be easily seen that (C2) implies (C1) $\forall t\in (0,T]$ since
      %  {\color{red} Why the following condition on $s$? Can we take it out?} and $s$ small enough depending only on function $\delta$,
\begin{equation}\label{weak_cond}  
        \TV{\beta^t}_s/\sqrt{s}  = \int_0^s \frac{1}{\sqrt s} d \TV{\b^t}_r \leq \int_0^s \frac{1}{\sqrt r} d \TV{\b^t}_r \leq \delta(s).
\end{equation}
%and 
%\[\int_{\vare}^t \frac{\abs{\b_r}}{r^{3/2}} dr \leq \int_{\vare}^t \frac{\TV{\b}_r}{r^{3/2}} dr  =  - \frac{2 \TV{\b}_t}{\sqrt t}  +  \frac{2 \TV{\b}_{\vare}}{{\sqrt{\vare}}} + \int_{\vare}^t \frac{2d\TV{\b}_r}{\sqrt r} \]
%converges when $\vare \to 0$. 
In addition, 
\[ \forall s,t \in [0,T],  \quad \TV{U}_{[s, t]}  \leq \d(|t-s|) \sqrt{|t-s|},\]
which shows that the $1/2$-H\"older norm of the driver converges uniformly to $0$ as the length of intervals goes to $0$.

         \item The results of Rohde-Marshall-Lind in \cite{RM}, \cite{Lind} shows that if the $1/2$-H\"older norm of the driver $U$ is less than $4$, the trace  is a $K$-quasi-slit, with $K$ going to $1$ as the H\"older norm approaches $0$. It is also not hard to see that $|U_{t+s} - U_t|/\sqrt{s}$ should converge to $0$ as $s\to 0$ at every $t$ to get a $C^1$ trace. One could ask whether the assumption that the $1/2$-H\"older norm is uniformly small on small intervals, e.g. given by condition (\ref{weak_cond}), is sufficient to imply $C^1$ trace. The answer is negative, and thus we require to put the stronger condition (C2).
         
 In fact, finite energy drivers (studied in \cite{semimartingale} and \cite{wang_loewner_energy}) are examples where the $1/2$-H\"older norm is uniformly small on small intervals but the trace is not necessarily $C^1$.  It is shown in \cite{wang_loewner_energy} that one can turn the trace to the right with angle $\t$, with an increasing driver whose energy is proportional to $\t^2$. By concatenating pieces of Loewner curves turning to the right  during short time with angle $1/n$ ($n = 1, \cdots, \infty$), one constructs a finite energy driver which generates an infinite spiral during finite time (see \cite{LoopEnergy} Section 4.2). This example satisfies (\ref{weak_cond}) but does not generate $C^1$ trace.
         
         We show a concrete driver $U$ where the above slow spiral happens at time $1$: $U$ is constant after time $1$, smooth on $[0,1]$ and for $s < 1/2$:         \[U_{1} - U_{1-s} =\b_s =  \int_{0+}^s \frac{ d r}{\sqrt{r}\log(r)}.\]
         The energy of $U$ on $[1-s, 1]$ is equal to
         \[\int_{0+}^s \dot{\b}_r^2 d r =  \int_{0+}^s \frac{1}{r (\log(r))^2} d r = \left[\frac{-1}{\log(r)}\right]_{0+}^s  \xrightarrow[]{s\to 0} 0.\]
         Thus, the condition (\ref{weak_cond}) holds:
         \[\TV{\beta}_s = |\b_s| = \left|\int_{0+}^s \dot{\b}_r d r\right| \leq \sqrt{s}\int_{0+}^s \dot{\b}_r^2 d r. \]
         This example fails at (C2) since
         \[\int_{0+}^s \frac{1}{\sqrt{r}} d \TV{\beta}_r = -\int_{0+}^s \frac{1}{r \log{r}} d r  = \infty.\]
          \item The result from \cite{carto-wong} shows that if $U \in C^\a$ with $\a > 1/2$, then its trace is in $C^{\a + 1/2} \subset C^1$. It is natural to ask whether $C^\a$ drivers satisfy (C2). Since the condition (C2) is on the total variation of the driver, it cannot cover all the $C^\a$ drivers. However, if the driver is monotonic, then it satisfies also (C2):
         \[ \int_0^{\vare} \frac{d \norm{\b^t}_r}{\sqrt{r}}  = \frac{\b^t_{\vare}}{\sqrt{\vare}} - \int_0^{\vare} \frac{\b^t_r}{r^{3/2}} d r \leq C \vare^{\a - 1/2} + \int_0^{\vare} C r^{\a - 3/2} d r \leq 2C \vare^{\a - 1/2}, \]
         for some constant $C >0$ independent of $t$.
        
     \end{itemize}

Now, we list some lemmas used in the proof of Proposition \ref{thm_(C1)}. We will use Lemma $4.2$ in \cite{lawlerbook} which is recalled below without proof for readers' convenience. 

\begin{lemma}[\cite{lawlerbook}, Lemma $4.2$] \label{lemma-from-lawler-book}
Let $X: [0,T) \to \mathbb{C}$ be a continuous function such that the right derivative \[X_{+}'(t) = \lim\limits_{h \to 0+ } \frac{X_{t+ h} - X_t }{h}\]exists everywhere and that $X_{+}'(t)$ is a continuous function. Then $X$ is continuously differentiable and $X'(t) = X_{+}'(t)$  for $t > 0$.
\end{lemma}

In view of the above lemma, establishing the right derivative turns out be relatively simpler to work with because of the directional nature of Loewner chains which is also reflected in the following lemma. 

Recall the definition of curve $\phi_s^t(w)$ as the solution of equation \eqref{eqn-of-phi} and \eqref{w=0-new} with $\phi_0^t(w) = w \in \mathbb{C}\setminus (0, \infty)$. 

The condition (C2) and Lemma \ref{uniquebranch} imply in particular that for all $0<s\leq t \leq T$,
\begin{equation}\label{lowerbound_im} 2\sqrt{s} \geq \Im \left(\sqrt{\phi^t_s (0)}\right) \geq \sqrt{4-\d(s)^2} \sqrt{s} \geq \sqrt{4-c^2} \sqrt{s} =: C\sqrt{s}.
\end{equation}

\begin{lemma}[Flow Property]\label{flow} If $U \in BV_{LR}[0,T]$, then for $s, t \in [0,T)$, $s  \le t$ and $h\geq 0$, 
\[ \phi_{s+h}^{t+h}(0) = \phi_s^t(\phi_h^{t+h}(0)).
\]

\end{lemma}  

\begin{proof}
Note that  
\[ \phi_{s+h}^{t+h}(0) = \phi_h^{t+h}(0) + 2\int_h^{s+h}\sqrt{\phi_{r}^{t+h}(0)}d\beta_r^{t+h} - 4s, \]
which implies that $s \mapsto \phi_{s+h}^{t+h}(0) $ is the solution of equation \eqref{eqn-of-phi} with the initial condition $w = \phi_h^{t+h}(0) $. Since equation \eqref{eqn-of-phi} has a unique solution, we conclude that $\phi_{s+h}^{t+h} (0)= \phi_s^t(\phi_h^{t+h}(0))$.
\end{proof}

\begin{proof}[Proof of Proposition \ref{thm_(C1)}]

We first establish the right derivative of curve $ \theta_t = \phi_t^t(0)$. Note that 
\[ \phi_h^{t+h}(0)= 2 \int_0^h \sqrt{\phi_r^{t+h}(0)}d\beta_r^{t+h} - 4h.\]
Since $|\Re \sqrt{\phi_r^{t+h}(0)}| \leq \TV{\beta^{t+h}}_r$ and $\Im \sqrt{\phi_r^{t+h}(0)} \leq 2 \sqrt{r}$, using condition (C2), we easily see that 
\[\lim\limits_{h \to 0+ } \phi_h^{t+h}(0)/h = -4 \]
This implies $\theta_{+}'(0) = -4$. For differentiability at $t_0 >0$, we will use Lemma \ref{flow}. Consider the curves \[ Z_s^{t_0, h} := \frac{\phi_{s+h}^{t_0+ h}(0)- \phi_s^{t_0}(0)}{\phi_h^{t_0 + h}(0)}=\frac{\phi_s^{t_0}(\phi_h^{t_0 + h}(0))- \phi_s^{t_0}(0)}{\phi_h^{t_0 + h}(0)}.\] 

 By (\ref{lowerbound_im}), one has
\[\Im \left(\sqrt{\phi_s^{t_0}(0)}\right) \geq C \sqrt{s}\]
and similarly
\[ \Im \left(\sqrt{\phi_s^{t_0}(\phi_h^{t_0 + h}(0))}\right) = \Im \left(\sqrt{\phi_{s+h}^{t_0 + h}(0)}\right)  \geq C \sqrt{s+ h}.\]

We claim that the family$\{Z^{t_0, h}\}_{ h>0}$ is equicontinuous for $h$ small enough. To see that, note 
\[ Z_v^{t_0, h} - Z_u^{t_0, h} = 2 \int_u^v \frac{Z_r^{t_0, h}}{\sqrt{\phi_r^{t_0}(\phi_h^{t_0 + h}(0))} +  \sqrt{\phi_r^{t_0}(0)}}d\beta_r^{t_0}\,.\]
Since the condition (C2) holds, Gronwall's inequality implies that family $\{Z^{t_0, h}\}$ is bounded, and thus its equicontinuity easily follows. Also, it follows from dominated convergence theorem that if $Z^{t_0}$ is any subsequential limit of $Z^{t_0, h}$ as $h \to 0+$, then 
\[Z_s^{t_0} = 1 + \int_0^s \frac{Z_r^{t_0}}{\sqrt{\phi_r^{t_0}(0)}}d\beta_r^{t_0}.\] 
Again using (C2) and similar proof as in Proposition \ref{existenceofphi}, we conclude that above equation has a unique solution, and thus $Z^{t_0, h}$ converges uniformly to $Z^{t_0}$. In fact, we can also write $Z^{t_0}$ in a closed form as 
\[ Z_s^{t_0} = \exp\biggl(\int_0^s\frac{1}{\sqrt{\phi_r^{t_0}(0)}}d\beta_r^{t_0}\biggr).\]

Then,
\[ \theta_{+}' (t_0) = \lim\limits_{h \to 0+}\frac{\phi_{t_0+h}^{t_0+ h}(0)- \phi_{t_0}^{t_0}(0)}{\phi_h^{t_0 + h}(0)}\frac{\phi_h^{t_0 + h}(0)}{h} 
 =\lim\limits_{h \to 0+}Z_{t_0}^{t_0, h} \frac{\phi_h^{t_0 + h}(0)}{h} = -4 \exp\biggl(\int_0^{t_0}\frac{d \beta_r^{t_0}}{\sqrt{\phi_r^{t_0}}}\biggr) \,.\]

Since $\theta_{+}'(0) = -4$, (C2) and (\ref{lowerbound_im}) imply that $\theta_{+}'$ is continuous at $t_0 =0$. For continuity at $t_0 > 0 $, let $t \in (t_0/2, 2t_0)$. Note that 
\[ \int_0^{t}\frac{1}{\sqrt{\phi_r^{t}}}d\beta_r^{t} = \int_0^{t_0} \frac{1}{\sqrt{\alpha_r^{t}}} d\beta_{r t/t_0}^t = \int_0^{\epsilon} \frac{1}{\sqrt{\alpha_r^{t}}} d\beta_{r t/t_0}^t + \int_{\epsilon}^{t_0} \frac{1}{\sqrt{\alpha_r^{t}}} d\beta_{r t/t_0}^t , \]
where $\alpha_s^t = \phi_{s t/t_0}^t$. Using again (\ref{lowerbound_im}), we see that $\a_{\vare}^t$ is uniformly bounded away from $0$ for each fixed $\epsilon >0$. Together with the proof in Proposition \ref{continuity in t}, we see that $1/\sqrt{\alpha_s^{t}}$ converges uniformly to $1/ \sqrt{\phi_r^{t_0}}$ on $[\epsilon,t_0]$ for any $\epsilon >0$. Thus, Lemma \ref{analysis-II} implies 
\[ \lim\limits_{t \to t_0 } \int_{\epsilon}^{t_0} \frac{1}{\sqrt{\alpha_r^{t}}} d\beta_{r t/t_0}^t = \int_{\epsilon}^{t_0} \frac{1}{\sqrt{\phi_r^{t_0}}} d\beta_{r}^{t_0}.\]
Since $\epsilon$ is arbitrary, using condition (C2), we conclude that $\theta_{+}'$ is continuous at $t_0 >0$ as well. Finally, Lemma \ref{lemma-from-lawler-book} implies that $\theta$ is continuously differentiable, which also implies $t \mapsto \gamma_{t^2}$ is
continuously differentiable. 
\end{proof}

\bibliographystyle{plain}

\end{document}